\documentclass[11pt]{amsart}
\usepackage{amsmath,amsthm,amssymb,mathrsfs}
\usepackage{microtype}
\usepackage{xspace}
\usepackage{cleveref,cite,url}
\usepackage{graphicx}
\usepackage[utf8]{inputenc}
\usepackage{enumitem}
\usepackage{dsfont} %needed for \mathsd{1}

\usepackage{xcolor}
\usepackage[nomargin,inline,draft]{fixme}

\newcommand{\RR}{\mathbb{R}}
\renewcommand{\SS}{\mathbb{S}}

\DeclareMathOperator{\Tr}{Tr}
\DeclareMathOperator{\rank}{rank}
\DeclareMathOperator{\Diag}{Diag}

\DeclareMathOperator{\cut}{cut}

\theoremstyle{plain}% default
\newtheorem{theorem}{Theorem}
\newtheorem{lemma}{Lemma}

\newtheorem{corollary}{Corollary}
\newtheorem{conjecture}{Conjecture}

\theoremstyle{definition}
\newtheorem{definition}{Definition}
\newtheorem{example}{Example}

\theoremstyle{remark}

\title[Optimal solutions and ranks in the max-cut SDP]{Optimal solutions and ranks\\ in the max-cut SDP}
\author{Daniel Hong}
\address{Skyline High School, Sammamish, WA 98075}
\email{greenhong18@gmail.com}
\author{Hyunwoo Lee}
\address{Interlake High School, Bellevue, WA 98008}
\email{hywlee03@gmail.com}
\author{Alex Wei}
\address{Interlake High School, Bellevue, WA 98008}
\email{alexwei@outlook.com}
\date{\today}

\begin{document}
\maketitle

\begin{abstract}
    The max-cut problem is a classical graph theory problem which is NP-complete.
    The best polynomial time approximation scheme relies on \emph{semidefinite programming} (SDP).
    We study the conditions under which graphs of certain classes have rank~1 solutions to the max-cut SDP.
    We apply these findings to look at how solutions to the max-cut SDP behave under simple combinatorial constructions.
    Our results determine when solutions to the max-cut SDP for cycle graphs are rank~1. We find the solutions to the max-cut SDP of the vertex~sum of two graphs. We then characterize the SDP solutions upon joining two triangle graphs by an edge~sum.

\emph{Keywords:} Max-cut, Semidefinite program, Rank, Cycle graphs, Vertex sum, Edge sum

\end{abstract}

\tableofcontents

\section{Introduction}

Consider a graph  $G=(V,E)$ with vertex set $V=[n],$ edge set \(E,\) and fixed weights $\{w_{ij}\}_{ij\in E}$ assigned to the edges. Consider partitions, or \emph{cuts,} of the vertex set \(V = V_1 \sqcup V_2.\) The \emph{max-cut problem} asks for the maximum possible sum of all weights between vertices on opposite sides of the cut. In other words, we wish to maximize \[\cut(V_1,V_2) := \sum_{i \in V_1, j \in V_2} w_{ij}\]across all partitions $V = V_1 \sqcup V_2$.

The max-cut problem is a key problem in theoretical computer science and operations research.
In particular, it is one of Karp's 21 NP-complete problems~\cite{karp1972reducibility}.
Furthermore, it has applications in physics and circuit design~\cite{baharona1988applications}. In statistical physics and magnetism, the Edwards-Anderson model for the ground states of spin glasses is an optimization problem in \(\pm1\)-variables and can be reduced to the max-cut problem. In Very-Large-Integrated-Scale (VLSI) circuit design, the problem of reducing the number of vias, or connections between layers of a circuit board, can be reduced to the max-cut problem.

Although the max-cut problem is NP-complete~\cite{karp1972reducibility}, a breakthrough paper of Goemans and Williamson~\cite{goemans1995improved} proved that the max-cut problem can be approximated in polynomial time up to a factor of 0.87854 using a technique known as \emph{semidefinite programming} (SDP).

A semidefinite program is an optimization problem where we minimize a linear function on the entries of a matrix subject to two types of constraints:
the matrix is \emph{positive semidefinite} (i.e., its eigenvalues are non-negative)
and its entries satisfy some linear equations.

Semidefinite programs have far-reaching applications \cite{vandenberghe1996semidefinite,wolkowicz2012handbook}, most notably in approximation algorithms, which we discuss in this paper.

The max-cut problem can be relaxed to a semidefinite program. The max-cut SDP requires the following definitions.

The \emph{Laplacian matrix} of a weighted graph $L(G,w)$ is the $n \times n$ symmetric matrix with entries
\begin{align*}
    L(G,w)_{ij} :=
    \begin{cases}
        -w_{ij} & \text{if } ij \in E \\
        \sum_k w_{ik} & \text{if } i=j\\
        0 & \text{otherwise.}
    \end{cases}
\end{align*}
We use the shorthand $X \succeq 0$ to indicate that the matrix $X$ is {positive semidefinite}, and we use $\bullet$ to represent the Frobenius inner product \(A \bullet B= \Tr(A^TB)\) for real square matrices \(A\) and \(B\) of the same dimension. Furthermore, \(\SS^n\) denotes the space of $n\times n$ symmetric matrices.

The max-cut SDP is presented below, the derivation of which we will present in \Cref{sec:background}.
\begin{align*}
    \max_{X \in \SS^n} \quad &\frac{1}{4} L(G,w) \bullet X \\
    \text{s.t.}\quad
    &X_{ii}=1 \text{ for } i \in [n]\\
     &X \succeq 0.
\end{align*}

In this paper, we study how the optimal solution to the max-cut SDP behaves under simple graph operations.
The operations covered in the paper are the \emph{vertex sum}, where two graphs are joined at a common vertex, and the \emph{edge sum}, where two graphs are joined at a common edge. These operations are special instances of a \emph{clique sum}, where two graphs are joined at a common clique \(K_n\) for some \(n.\)

In this paper we pay special attention to the rank of the optimal solution obtained.
In general, low rank solutions to the max-cut SDP are desirable.
In particular, if the optimal solution is rank~1 then we may recover an exact optimal solution of the max-cut problem.
Our current results characterize when cycle graphs have rank~1 solutions to their max-cut SDPs.
We also study the behaviour for random weights.
For random normally distributed weights, we experimentally estimate the probability that the optimal solution is rank~1.

\subsection*{Related work}
The max-cut SDP was first studied in~\cite{delorme1993laplacian},
where the authors showed promising experimental results.
It gained major attention after the seminal work of Goemans and Williamson~\cite{goemans1995improved},
where they introduced a rounding technique
capable of producing a cut which is optimal up to a factor of~0.87854.
The gives the best known polynomial time approximation algorithm to the max-cut problem.

Having a low rank solution has many important computational and theoretical consequences.
Indeed, if the optimal solution has low rank then one can take advantage of faster optimization routines~\cite{burer2002rank},
and the approximation factor can be also be improved~\cite{avidor2005rounding,de2020integrality}.

The rank of the optimal solution matrix is closely related to the geometric structure of the SDP.
The geometry of the max-cut SDP has been investigated in
\cite{laurent1995positive,laurent1996facial,silva2018strict}.
The paper \cite{nagy2014forbidden} leverages geometric information to provide upper bounds on the rank of the solution.

It is possible to construct more sophisticated SDPs for the max-cut problem that rely on a technique known as the \emph{SOS method}~\cite{blekherman2012semidefinite}.
It is an open problem whether such sophisticated SDPs can improve the approximation factor.
This problem has received quite a lot of interest in recent years,
and it is closely related to a twenty-year old conjecture from computer science known as the \emph{unique games conjecture}~\cite{khot2007optimal}.
We do not consider these advanced SDPs here.

\subsection*{Structure}
The structure of this paper is as follows.
In \Cref{sec:background}, we derive and state the max-cut semidefinite program and provide definitions related to weighted graphs and clique sums of graphs.
In \Cref{sec:ranks}, we discuss experimental findings on the ranks of solutions to the max-cut SDP.
In \Cref{sec:ranksthm}, we discuss theorems which characterize the ranks of solutions to the max-cut SDP, as well as solutions to the max-cut SDP for vertex sums and an edge sum of triangles.
In \Cref{sec:proofs}, we provide proofs for theorems discussed in \Cref{sec:ranksthm}.
In \Cref{sec:future work}, we discuss unresolved conjectures not covered in \Cref{sec:proofs} and potential future avenues for research.

\section{Background}\label{sec:background}
\subsection{Semidefinite programs}

Let \(C \in \SS^n\) be an \(n\times n\) cost matrix.
Consider \(m\) constraint matrices \(A_1,A_2, \ldots, A_m\in \SS^n\),
as well as a constraint vector \( b \in \RR^m.\)
A \emph{semidefinite program} is an optimization problem of the form
\begin{align*}
    \max_{X} \quad &C \bullet X \\
    \text{s.t.} \quad & A_i \bullet X = b_i ~~ \forall 1\le i \le m \\
    &X\succeq 0.
\end{align*}

For each semidefinite program, there exists an associated \emph{dual semidefinite program} The dual SDP is formulated as follows (here \(b, A_i\) and \(C\) are the same as above):
\begin{align*}
    \min_{y,S} \quad &b^Ty \\
    \text{s.t.} \quad & S = \sum_{i=1}^n y_iA_i - C\\
    & S \succeq 0.
\end{align*}
We denote the primal value \(p^*\) to represent the maximal value achieved by the SDP across its domain, and we similarly define \(d^*\) as the dual optimal value. It is always the case that the primal value \(p^*\) is at most the dual value \(d^*,\) or $p^* \leq d^*$.
This is known as \emph{weak duality}.
In case that the two values agree, or $p^* = d^*$,
we say that \emph{strong duality} holds.
It is known that strong duality holds under mild assumptions,
see e.g. \cite[Thm~3.1]{vandenberghe1996semidefinite}. 

\subsection{Max-cut SDP}

In this section we explain the connection between the max-cut problem and its SDP relaxation.

Given a cut \((V_1,V_2)\), consider assigning \(1\) to all vertices in \(V_1\) and \(-1\) to all vertices in \(V_2\).
Then we can treat a cut as a vector \(x\) with \(x_i^2=1\) for all \(i\). Now, the quantity \(1-x_ix_j,\) for any edge \(ij \in E,\) will be \(0\) if \(i\) and \(j\) are in the same group, and \(2\) otherwise. Thus, we arrive at
\begin{align*}
    \cut(V_1,V_2)
    = \sum_{i \in V_1, j \in V_2} \!\!\!\!w_{ij}
    &=\frac{1}{2}\sum_{ij \in E} w_{ij}(1{-}x_ix_j)
    =\frac{1}{4} \sum_{i,j} L(G,w)_{ij}x_ix_j.
\end{align*}
Let \(X=xx^T \in \SS^n\).
Note that \(X\) is rank~1, positive semidefinite, and has all diagonal entries equal to \(1\). Moreover, all such matrices \(X\) satisfying those three conditions can be rewritten as \(X=xx^T\) for some vector \(x\) with \(x_i^2=1\) for all \(i.\)

Thus, the max-cut problem can be rephrased as
\begin{equation*}
\begin{aligned}
\max_{V_1,V_2} \;\;\cut(V_1,V_2)
\end{aligned}
=
\begin{aligned}
\max_{x} \;\;&\tfrac{1}{4}\, x^T L(G,w) x\\
\;\text{ s.t }\;\;&x_i^2 = 1\;\;\forall i
\end{aligned}
=
\begin{aligned}
\max_{X} \;\;&\tfrac{1}{4}\, L(G,w) \bullet X\\
\;\text{ s.t }\;\;&X_{ii} = 1\;\;\forall i\\
&X \succeq 0, \; \rank X = 1.
\end{aligned}
\end{equation*}
The last optimization problem involves the nonconvex constraint $\rank X=1$.
We can relax this problem to a semidefinite program by getting rid of the rank~1 constraint. By doing so, we arrive at the primal max-cut SDP.

\begin{definition}
Let \(C=\frac{1}{4}L(G,w)\). The \emph{primal} max-cut SDP is the following relaxation of the max-cut problem:
\begin{align*}
    \max_{X \in \SS^n} \quad &C \bullet X \\
    \text{s.t.}\quad
    &X_{ii}=1 \text{ for } i \in [n]\\
     &X \succeq 0.
\end{align*}
\end{definition}

Henceforth, the max-cut SDP will refer to the primal max-cut SDP.
Note that if the optimal solution $X$ of the SDP has rank~1,
then we may write it in the form $X=x x^T$,
and hence we may recover the optimal cut.
We say that the relaxation is \emph{exact} if this happens.

We now introduce the dual of this semidefinite program. Note that in the situation above, the constraint variables \(A_i\) and \(b_i\) correspond to each diagonal entry on the matrix. In particular, \(A_i\) is a matrix with \(ii\)th entry \(1\) and all other entries \(0\), while \(b_i=1\). We can then define the dual for the max-cut~SDP.

\begin{definition}
The \emph{dual} max-cut SDP is as follows:
\begin{align*}
    \min_{y\in \RR^n, S \in \SS^n}\quad &\sum y_i \\
    \text{s.t.}\quad S&=\Diag(y)-C \\
    S &\succeq 0.
\end{align*}
\end{definition}

We say a matrix \(X\) is \emph{primal feasible} if it satisfies all the constraints of the primal SDP. Similarly, we say a matrix \(S\) is \emph{dual feasible} if it satisfies all the constraints of the dual SDP.

The following is a well known theorem that characterizes the optimal solutions to any SDP satisfying strong duality. Note that all max-cut SDPs are known to satisfy strong duality.

\begin{theorem}[{\cite[eq.(33)]{vandenberghe1996semidefinite}}]\label{thm:optcondition}
Suppose we have two matrices $\bar X, \bar S$. The following three statements are satisfied simultaneously if and only if \(\bar X \) and \(\bar S\) are optimal solutions to the primal and dual max-cut SDP, respectively:
\begin{itemize}
    \item \(\bar X\) is primal feasible
    \item \(\bar S\) is dual feasible
    \item Complementary slackness, or \(\bar X \bar S = 0.\)
\end{itemize}

\end{theorem}

\subsection{Clique sums of graphs}

In this paper, we study ranks of solutions to the primal max-cut SDP in relation to clique-sums. A \emph{clique-sum} of two graphs \(G\) and \(H\) which both contain a clique graph \(K_n\), and we join \(G\) and \(H\) along this clique graph \(K_n\) to form a new graph \(F\), which has two vertex sets \(V_1\) and \(V_2\) such that the following are true:
\begin{itemize}
    \item The induced subgraph of \(F\) on \(V_1\) is isomorphic to \(G,\)
    \item The induced subgraph of \(F\) on \(V_2\) is isomorphic to \(H,\)
    \item The induced subgraph of \(F\) on \(V_1 \cap V_2\) is isomorphic to \(K_n\), and
    \item \(V_1 \cup V_2\) is the vertex set of \(H\).
\end{itemize}
We then define a \emph{vertex sum} as a clique-sum with the clique \(K_1,\) and an \emph{edge sum} as a clique-sum with the clique \(K_2\).

\section{Experiments on rank of max-cut SDP}\label{sec:ranks}
Rank is important to the max-cut problem because lower rank solutions will be able to yield better approximate solutions to the max-cut SDP. We have already noted that a rank~1 solution will exactly find the solution to the max-cut problem, while it has been shown that rank~2 solutions yield better approximation algorithms for the max-cut problem.

Given a weighted graph $G$, the set of optimal primal (or dual) solutions might not be unique.
We let $r_P(G)$ be the largest rank among all possible primal optimal solutions.
Similarly, we may define $r_D(G)$ to be the largest rank among all possible dual optimal solutions.
We say that strict complementarity holds if $r_P(G) + r_D(G) = n$.
It is known that strict complementarity holds generically \cite{alizadeh1997complementarity}.
Hence, we restrict our attention to the primal rank $r_P(G)$.

We study the primal ranks \(r_P(G)\) for random families of weighted graphs.
For a given graph \(G=([n],E)\), we consider two probability distributions on the weights $\{w_{ij} : ij \in E\}$.
In the first probability distribution, each weight $w_{ij}$ is an independent Gaussian random variable.
The second probability distribution is similar,
except that we force the weights $w_{ij}$ to be nonnegative.

The probabilities shown in \Cref{tab:probabilities} are the experimental values for the probability that the solution to the primal max-cut SDP of a graph returns a certain rank.

To obtain this data, we generated 1000 different weights for each graph \(G=(V,E)\) corresponding to random points on the \(|E|\)-dimensional unit sphere (for arbitrary weights), then 1000 different weights corresponding to random points on the \(|E|\)-dimensional unit sphere in the first orthant (for positive weights). We then counted the number of times that each of the ranks was obtained.

Known algorithms will return the highest rank solution, so graphs which yield multiple solutions or solutions of multiple different ranks will not be shown in the following table.

In \Cref{tab:probabilities}, ``Diamond'' represents a \(K_4\) graph with the edge removed, ``Butterfly'' is the vertex sum of \(K_3\) and \(K_3,\) and the ``Fish'' graph is the vertex sum of \(K_3\) and \(C_4.\)

\begin{table}[h]
\caption{Rank distribution of optimal solutions}
\label{tab:probabilities}
\begin{center}
\begin{tabular}{ |c|c|c|c|c|c|c| }
 \hline
  & \multicolumn{3}{| c |}{Arbitrary Weights} & \multicolumn{3}{| c |}{Positive Weights} \\
 \hline
 Graph & Rank~1 & Rank~2 & Rank~3 & Rank~1 & Rank~2 & Rank~3\\
 \hline
 \(K_3\) & 85\% & 15\% & 0\% & 69\% & 31\% & 0\% \\
 \(C_4\) & 77\% & 23\% & 0\% & 100\% & 0\% & 0\% \\
 Diamond & 71\% & 29\% & 0\% & 65\% & 35\% & 0\% \\
 \(C_5\) & 73\% & 27\% & 0\% & 45\% & 55\% & 0\% \\
 Butterfly & 72\% & 25\% & 3\% & 50\% & 42\% & 8\% \\
 \(C_6\) & 70\% & 30\% & 0\% & 100\% & 0\% & 0\% \\
 Fish & 62\% & 34\% & 4\% & 69\% & 31\% & 0\% \\
 \hline
\end{tabular}
\end{center}
\end{table}

We know a few facts about the ranks of the optimal matrices. In particular, there always exists a primal matrix $\bar X$ with rank $r$ such that $\binom{r+1}{2} \leq n$, and there always exists a dual matrix $\bar S$ with rank $s$ such that $s \leq n\!-\!1$, see \cite[Thm~2.1]{pataki1998rank}.

However, there can also exist solutions with rank greater than the primal value $p$. For instance, the butterfly graph has an approximately 2\% chance of returning a rank~3 primal solution, despite the fact that there must exist an optimal solution of rank~2 or lower. When this occurs, we know that the primal solution is not unique. For most graphs, there will be a unique solution to the primal matrix. 

\section{Solutions and Ranks for Particular Classes of Graphs}\label{sec:ranksthm}

\subsection{Rank~1 solutions for cycles}

We aim to understand the conditions in which the rank of a primal solution is 1 for certain graphs.

We will prove the following theorem regarding the rank of solutions to the primal max-cut SDPs of cycle graphs in \Cref{sec:proofs}.

\begin{theorem}\label{thm:cyclerank1}
Consider a cycle graph \(C_n\) with \(V=[n]\) and \(E\) consisting of all edges \((i,i+1)\) with indices taken modulo \(n\). For simplicity, let \(w_i\)  denote the weight of the edge \((i,i+1)\). There exists a rank~1 solution to the primal SDP \(L(G,w)\) if and only if at least one of the following statements is true:
\begin{itemize}
    \item There are an even number of positively weighted edges in the graph.
    \item There exists a weight \(w_m\) such that
    \[\frac{1}{|w_m|} \ge \sum_{i \neq m} \frac{1}{|w_i|}.\]
\end{itemize}
\end{theorem}
The following corollary also holds:
\begin{corollary}\label{cor:uniquecycsol}
If a cycle graph has a rank~1 solution to its corresponding max-cut primal SDP, then the rank~1 solution is unique.
\end{corollary}
As an application of the above theorem, we calculate the probability that the \(K_3\) graph has a rank~1 primal solution in the following example.

\begin{example}
The probability that a a triangle, or \(K_3\), with weights randomly chosen from a standard distribution, has a rank~1 primal solution, is \(\frac{6-2\sqrt{3}}{3}\). When restricted to random positive weights, this probability is \(\frac{9-4\sqrt{3}}{3}\).

Suppose we choose a random vector of weights \((w_1,w_2,w_3)\in \RR^3\). Without loss of generality, we can normalize this vector so that \(w_1^2+w_2^2+w_3^2=1\). We claim that the condition of \Cref{thm:cyclerank1} is satisfied if and only if either \(w_1+w_2+w_3,w_1-w_2-w_3,-w_1+w_2-w_3,\) or \(-w_1-w_2+w_3\) is at most \(-1\). Indeed, note that if there are an even number of positively weighted edges, one of these \(4\) values has \(3\) negative terms, and thus as \(|x|>x^2\) for all \(x^2 \le 1\), it satisfies this condition. Otherwise, note that
\begin{align*}
    \frac{1}{|w_1|} \ge \frac{1}{|w_2|}+\frac{1}{|w_3|} &\Longleftrightarrow |w_2w_3| \ge |w_1w_2|+|w_1w_3| \\
    &\Longleftrightarrow (|w_1|-|w_2|-|w_3|)^2 \ge 1.
\end{align*}

Looking at each octant individually, we see that the resulting signs will define the four inequalities as described. Using calculus, it can be shown that the probability that one of the four inequalities is satisfied is \(p=\frac{1}{2}-\frac{\sqrt{3}}{6}\). As a result, the probability the solution to the primal max-cut SDP of the triangle with random weights is rank~1 is \(4p=\dfrac{6-2\sqrt{3}}{3} \approx 0.8453.\) Since \(0 < \rank \bar X \le 3-\rank \bar S \le 2\), then the remaining probability is the probability the primal is rank~2.

Note that when we randomly choose weights, we have a \(\frac{1}{2}\) chance of having an even number of positive weights, thus this total probability is the average of \(1\) and the probability for an odd number of positive weights. As a result, when we restrict our random weights to be all positive, the probability is
\[2\left(\frac{6-2\sqrt{3}}{3}\right)-1 = \frac{9-4\sqrt{3}}{3} \approx 0.6905.\]
\end{example}

\subsection{Max-cut SDP for a vertex sum}

We have characterized the conditions where cycle graphs are rank~1. We now characterize the probability distributions of primal solutions graphs that are the vertex sum of two graphs, in terms of the distributions of the subgraphs.

Say \(p_1\) and \(p_2\) be the probabilities that the rank of the solution to the primal max-cut SDP for \(K_3\) with arbitrary weights is 1 and 2, respectively.

According to the experimental values in \Cref{tab:probabilities}, the probability for arbitrary weights that the rank of the primal matrix is 1 is \(p_1^2,\) the probability it is rank~2 is \(2p_1p_2,\) and the probability it is rank~3 is \(p_2^2.\)

We obtain similar results for the Fish graph. Furthermore, equivalent results hold for the normally distributed nonnegative weights. 

Suppose we have two graphs \(G_1=(V_1,E_1),G_2=(V_2,E_2),\) where \(L(G_1)\) has primal-dual solution pair \(X_1,S_1\), and \(L(G_2)\) has primal-dual solution pair \(X_2, S_2\). Without loss of generality suppose we sum the graphs along the last element of \(V_1\) and the first of \(V_2\). Let \(x_1\) be the final column of \(X_1\), and as the final entry of \(x_1\) is \(1\), define \(y_1\) so that \(x_1=\begin{bmatrix} y_1 \\ 1
\end{bmatrix}\). Similarly, let \(x_2=\begin{bmatrix}
1 \\ y_2
\end{bmatrix}\) be the first column of \(X_2\). Finally, define \(Y_1,Y_2\) so that

\[X_1=\begin{bmatrix}
Y_1 & y_1 \\ y_1^T & 1
\end{bmatrix},
X_2=\begin{bmatrix}
1 & y_2^T \\ y_2 & Y_2
\end{bmatrix}.\]

The above observations regarding the probability distributions motivates the following theorem.

\begin{theorem}\label{thm:vertexsum}
The primal-dual pair

\[X'=\begin{bmatrix}
Y_1 & y_1 & y_1y_2^T \\
y_1^T & 1 & y_2^T \\
y_2y_1^T & y_2 & Y_2
\end{bmatrix}, \qquad 
S' = \begin{bmatrix}
S_1 & 0 \\ 0 & 0
\end{bmatrix} + \begin{bmatrix}
0 & 0 \\ 0 & S_2
\end{bmatrix}\]
is optimal.
Furthermore, \(\rank X'= \rank X_1 + \rank X_2 - 1\).
\end{theorem}

A direct consequence of this result is that we now have well-defined solutions to the primal and dual max-cut SDPs of two graphs summed at a vertex. The following corollary shows that in fact the max-cut problem for a vertex sum graph also has well-defined ranks for its solutions.

\begin{corollary}\label{cor:vtxsumranks}
If \(G'\) is a vertex sum of \(G_1\) and \(G_2\) along an element of \(V_1\) and an element of \(V_2\), then there exists a primal solution to \(L(G')\) with rank at most \(\max(\rank(X_1), \rank(X_2)).\)
\end{corollary}

The proof of the above theorem and corollary is presented in \Cref{sec:proofs}.
To illustrate how this theorem helps us further determine the probabilities, we provide the following example.

\begin{example}
The probabilities for the butterfly graph primal SDP returning a solution of rank~1,~2,~and~3 are in fact \(p_1^2,2p_1p_2,\) and \(p_2^2\), respectively, where \(p_1=\dfrac{6-2\sqrt{3}}{3}\) and \(p_2=1-p_1=\dfrac{2\sqrt{3}-3}{3}\).

Note that the best known algorithms return the highest rank solutions. As a result, a rank~1 solution will be returned if and only if the two summed \(K_3\) graphs are both yield rank \(1\) primal solutions, a rank~3 solution if and only if both yield rank \(2\) solutions, and a rank~2 solution otherwise. This gives us the probabilities desired.
\end{example}

\subsection{Max-cut SDP for a diamond graph}

In addition to our study of vertex sums, we also studied edge sums. In particular, we focused on the diamond graph, which is the edge sum of two \(K_3\) graphs.

Suppose that we have a \textit{diamond} graph with \(V=\{1,2,3,4\}\) and \linebreak \(E=\{(1,2), (1,3), (2,3), (2,4), (3,4)\}\). Let \(X^*,S^*\) be a primal-dual solution pair to the max-cut SDP for \(G\). Let \(G_1\) and \(G_2\) be the induced subgraphs of \(G\) on vertex sets \(\{1,2,3\}\) and \(\{2,3,4\}\) respectively with same weight vector as \(G\).

\begin{center}
    \includegraphics[width=3in]{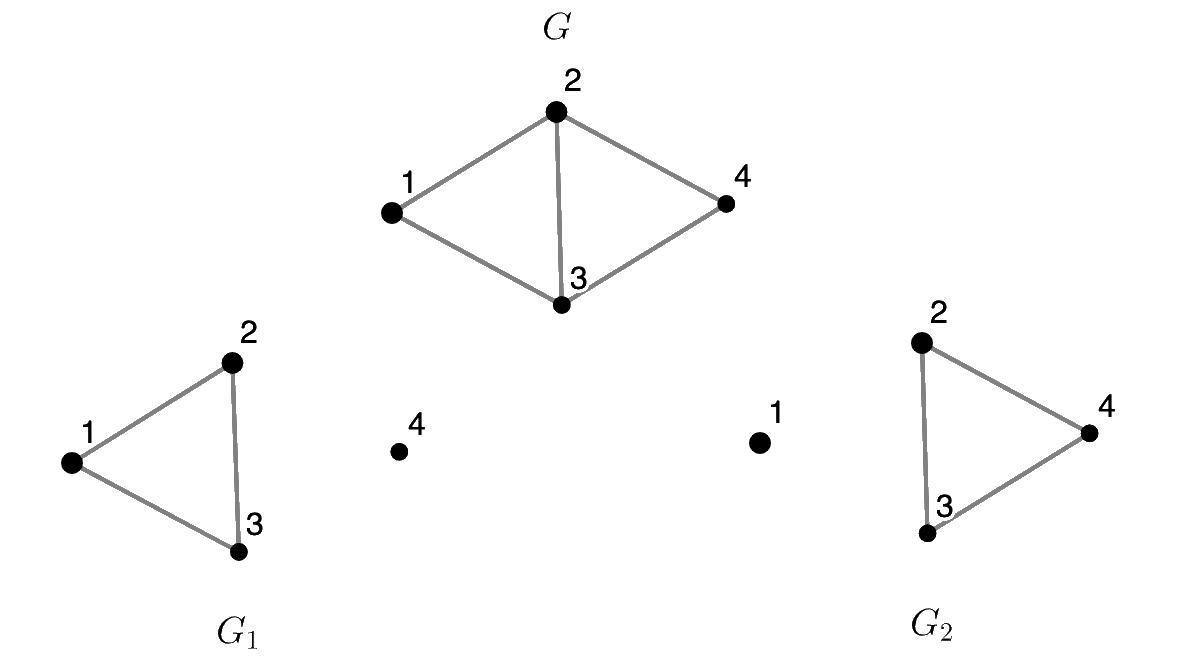}
\end{center}

We prove the following two theorems about diamond graphs.

\begin{theorem}\label{thm:diamond1}
Assume that \(G_1\) and \(G_2\) both yield rank~1 optimal primal solutions which agree on the joined edge between vertices \(2\) and \(3\). Then the following are equivalent:
\begin{itemize}
    \item \(w_{23} \le \min(\frac{1}{1/w_{12}+1/w_{13}},\frac{1}{1/w_{24}+1/w_{34}}).\)
    \item The matrix \(X\), where \(x=[-1,1,1,-1]\) and \(X=xx^T\), is an optimal primal solution.
\end{itemize}
Also, the above items imply that the optimal dual solution \(S^*\) is of the form
    \[S^*=\begin{bmatrix}
    S_1 & 0 \\
    0 & 0
    \end{bmatrix}+
    \begin{bmatrix}
    0 & 0 \\
    0 & S_2
     \end{bmatrix}+\frac14
    \begin{bmatrix}
    0 & 0 & 0 & 0 \\
    0 & -w_1 & w_1 & 0 \\
    0 & w_1 & -w_1 & 0 \\
    0 & 0 & 0 & 0
    \end{bmatrix}.\]
\end{theorem}

\begin{theorem}\label{thm:diamond2}
Assume that \(G_1\) and \(G_2\) both yield rank~1 optimal primal solutions which agree on the joined edge between vertices \(2\) and \(3\). Then the following are equivalent:
\begin{itemize}
    \item \(w_{23} \ge \frac{1}{|1/w_{12}-1/w_{13}|}+\frac{1}{|1/w_{24}-1/w_{34}|}.\)
    \item The matrix \(X\), where \(x=[\epsilon_1,-1,1, \epsilon_2]\) and \(X=xx^T\), is an optimal primal solution for some \(\epsilon_1,\epsilon_2 \in \{-1,1\}\).
\end{itemize}
The above items imply that the optimal dual solution \(S^*\) is of the form
    \[S^*=\begin{bmatrix}
    S_1 & 0 \\
    0 & 0
    \end{bmatrix}+
    \begin{bmatrix}
    0 & 0 \\
    0 & S_2
    \end{bmatrix}+\frac14
    \begin{bmatrix}
    0 & 0 & 0 & 0 \\
    0 & w_1 & w_1 & 0 \\
    0 & w_1 & w_1 & 0 \\
    0 & 0 & 0 & 0
    \end{bmatrix}.\]
\end{theorem}

\section{Proofs}\label{sec:proofs}

\subsection{General Lemma on Rank~1 optimal solutions}
From \Cref{thm:optcondition}, we can obtain that if \(x\) is the ones vector $(1, \ldots, 1)\in \RR^n,$ then \(\bar{S}x=0 \) if and only if \(\bar{S}=-\frac{1}{4}L.\) Therefore, the ones matrix \(xx^T\) is an optimal solution to the primal max-cut SDP if and only if and only if \(\bar{S} \succeq 0\) when \(\bar{S}x=0\), which is equivalent to \(-L(G,w)\succeq 0.\)

We can generalize for any vector \(x=(\pm 1, \ldots, \pm 1)\in \RR^n.\) This result will be helpful for the proofs in following subsections.
\begin{lemma}\label{lem:rankone}
Suppose a graph \(G\) has weight vector \(w\) and the primal SDP of the matrix \(L(G,w)\) has a rank~1 optimal solution. There there exists a vector \(x\) with \(x_i^2=1\) for all \(i,\) such that for  edge weights \(\overline{w}\) defined with \(\overline{w_{ij}}=-x_ix_jw_{ij} ~ \forall ~ i,j \in [n]\), the matrix \(S'=L(G,\overline{w})\) is positive semi-definite.
\end{lemma}

\begin{proof}
Suppose the primal SDP returns a solution \((\bar X,\bar S)\) with \(\bar X\) rank~1. By definition, we know \(\bar X=xx^T\) for some vector \(x\) with \(x_i=\pm 1\) for all \(i\). Since \(x\) is the first column of \(X\), then \(Sx=0\). Define \(S'\) with \(S'_{ij}=4x_ix_j \bar S_{ij}\). Then for all edges \(i \in E(G)j\), we just have \(S'_{ij}=x_ix_j \cdot 4S_{ij}=-x_ix_jw_{ij}\). It is evident that \(S'_{ij}=0\) for \(ij\) not connected. Let \(\bar S_i\) be the \(i\)th column of \(\bar S\), and note by definition \(x^TS_i=0\). Now,
\[\sum_{j=1}^n S'_{ij}= \sum_{j=1}^n 4x_ix_j \bar S_{ij}=4x_i(x^T \bar S_i)=0.\]
Thus, \(S'\) is indeed the Laplacian matrix with the weights as claimed. Finally, note that for any vector \(y\), we have
\[y^T S' y = \sum_{1 \le i,j \le n} 4y_iy_jx_ix_j\bar S_{ij}.\]
We set a vector \(z\) with \(z_i=x_iy_i.\) Because \( \bar S \succeq 0,\) we have \(z^T \bar S z \ge 0,\) and thus \(y^T S' y \ge 0\). As a result, \(S'\) is positive semidefinite.
\end{proof}

This proof tells us if \(\bar X\) is rank~1 and is an optimal solution to the primal max-cut SDP for some graph, then some transformation of \(L\) is positive semi-definite. More practically, the set of constraints on \(L\) to have a rank~1 optimal primal solution is orthogonal to the condition for which the all ones matrix is an optimal primal solution.

\subsection{Rank~1 solutions for cycles}

\begin{proof}[Proof of \Cref{thm:cyclerank1}]
By \Cref{lem:rankone}, we know that a rank~1 solution will lead to the exact optimal cut \(x\) with \(x_i^2=1\) for all \(i\). For convenience let \(x_{n+1}=x_1\). Suppose that \(w_kx_kx_{k+1}\) is positive for at least two different values of \(k\), say \(k_1<k_2\). Then consider the vector \(x'\) with \(x'_k=-x_k\) for all \(k_1 < k \le k_2\) and \(x'_k=x_k\) otherwise. It is clear that the total sum of \(w_kx_kx_{k+1}\) will decrease, thus leading to a better primal solution. As a result, \(w_ix_ix_{i+1}\) is positive for at most one value of \(i\).

If \(-w_ix_ix_{i+1} \ge 0\) for all values of \(i\), then we can take the product of all of these constraints to get
\[\prod_{i=1}^n -w_i = \prod_{i=1}^n -w_ix_ix_{i+1} \ge 0.\]
Thus there are an even number of positive \(w_i\) in this case. In fact, the converse is true: if there are an even number of positive \(w_i\), it is always possible to choose an \(x\) so that \(w_ix_ix_{i+1}\) is negative for all \(i\).

Now, without loss of generality suppose that \(x=(1, \ldots, 1)\in \RR^n,\) so that the matrix \(S'=-L(G,w)\) is positive semidefinite by \Cref{lem:rankone}. \(S' \succeq 0\) if and only if the determinant of each principal submatrix is positive. Note that \(S'\) has the form
\[
\begin{bmatrix}
-(w_n+w_1) & w_1 & \dots & 0 & w_n\\
w_1 & w_1+w_2 & \dots & 0 & 0\\
\vdots & \vdots & \ddots & \vdots & \vdots \\
0 & 0 & \dots & -(w_{n-2}+w_{n-1}) & w_{n-1} \\
w_n & 0 & \dots & w_{n-1} & -(w_{n-1}+w_n) \\
\end{bmatrix}.
\]
By induction, we can show that the determinant of the lower-right \(m \times m\) submatrix of \(S'\) is the \(m\)th symmetric sum of the numbers \(-w_{n-m},-w_{n-m+1},\ldots,\) \(-w_n\) for \(1 \le m <n\). Call this symmetric sum \(d_m\), and note that for \(m ge 2\), \[d_m=-w_{n-m}d_{m-1}+(-1)^mw_{n-m+1}\ldots w_{n-1}w_n.\]

The base cases of \(m=1\) and \(m=2\) are easily verified, giving determinants of \(d_1\) and \(d_2\). Assuming that the hypothesis is true for \(m\) and \(m-1\) with \(m \ge 2\), we can evaluate the determinant of the lower-right \(m+1 \times m+1\) matrix to get
\begin{align*}
    &-(w_{n-m-1}+w_{n-m})d_m+w_{n-m}^2d_{m-1} \\
    =& -w_{n-m-1}d_m-w_{n-m}(d_m-w_{n-m}d_{m-1}) \\
    =& -w_{n-m-1}d_m-w_{n-m}(-1)^mw_{n-m+1}\ldots w_{n-1}w_n \\
    =& \quad d_{m+1}.
\end{align*}

We only need to consider determinants of principal submatrices that correspond to connected components of the cycle graph, because other can be expressed in terms of those. Note that if all \(w_i\) are negative, we automatically satisfy all the necessary inequalities. Otherwise, we have at least one constraint \(d_{n-1} \ge 0\) and note that for \(w_i\) all nonzero it is equivalent to
\[(-w_1)(-w_2)\ldots (-w_n)\left(\frac{-1}{w_1}+\ldots+\frac{-1}{w_n}\right) \ge 0.\]
As shown earlier, at most one \(w_i\) can be positive, so if exactly one \(w_i\) is positive, then we must have
\[\frac{1}{w_1}+\ldots+\frac{1}{w_n} \ge 0,\]
as desired. All other constraints follow similarly from this single constraint; thus, this constraint is both necessary and sufficient to describe when the cycle graph has a solution with \(x\) as the all ones vector, where not all weights are negative. For any other vector \(x\), the weights \(w_i\) will simply be replaced instead by the values \(x_ix_{i+1}w_i,\) and and the final necessary and sufficient constraint is that there exists a weight \(w_m\) such that
\[\frac{1}{|w_m|} \ge \sum_{i \neq m} \frac{1}{|w_i|}.\]
\end{proof}

\begin{proof}[Proof of \Cref{cor:uniquecycsol}]
If the first condition (an even number of positive edges) is satisfied, then the vector \(x\) is the unique vector (up to multiplication by \(-1\)) so that every \(w_ix_ix_{i+1}\) is negative.

Otherwise, we have some weight \(w_m\) such that
\[\frac{1}{|w_m|} \ge \sum_{i \neq m} \frac{1}{|w_i|}.\]
Note that clearly, \(|w_m|<|w_i|\) for all \(i \neq m\), and as a result \(m\) is unique. The the unique solution to the max-cut problem (and thus the unique solution to the max-cut SDP) is the vector \(x\) such that \(w_ix_ix_{i+1}\) is negative for all \(i \neq m\), and positive for \(i=m\).

In both cases, a rank~1 solution for the cycle graph must be unique.
\end{proof}

\subsection{Max-cut SDP for a vertex sum}
\begin{proof}[Proof of \Cref{thm:vertexsum}]
Due to the dual feasibility of \(S_1,S_2\) we know \(S'\) is a sum of two positive semidefinite matrices, so evidently \(S'\) is dual feasible. Now note \(X'=\begin{bmatrix}
X_1 & x_1y_2^T \\ y_2x_1^T & Y_2
\end{bmatrix} = \begin{bmatrix}
Y_1 & y_1x_2^T \\ x_2y_1^T & X_2
\end{bmatrix}\).
From here, we see
\begin{align*}
    X'S' &= \begin{bmatrix}
X_1 & x_1y_2^T \\ y_2x_1^T & Y_2
\end{bmatrix} \begin{bmatrix}
S_1 & 0 \\ 0 & 0
\end{bmatrix} + \begin{bmatrix}
Y_1 & y_1x_2^T \\ x_2y_1^T & X_2
\end{bmatrix} \begin{bmatrix}
0 & 0 \\ 0 & S_2
\end{bmatrix} \\
&= \begin{bmatrix}
X_1S_1 & 0 \\ y_2x_1^TS_1 & 0
\end{bmatrix} + \begin{bmatrix}
0 & y_1x_2^TS_2 \\ 0 & X_2S_2
\end{bmatrix} = \begin{bmatrix}
0 & 0 \\ 0 & 0
\end{bmatrix}.
\end{align*}
From the fact that \(X_1S_1=X_2S_2=0.\)

Now, let \(V\) be a matrix with \(n\) columns and column size equal to the rank of \(X_1\) such that \(V^TV=X_1\). Let \(v\) be the final column of \(v\). Then we know \(V^Tv\) is the last column of \(X_1\), which is \(x_1\), and similarly \(v^Tv=1\). Note that
\[\begin{bmatrix}
V^T \\ y_2v^T
\end{bmatrix} \begin{bmatrix}
V & vy_2^T
\end{bmatrix} = \begin{bmatrix}
V^TV & y_2v^TV \\ V^Tvy_2^T & y_2v^Tvy_2^T
\end{bmatrix} = \begin{bmatrix}
X_1 & y_2x_1^T \\ x_1y_2^T & y_2y_2^T
\end{bmatrix}.\]
Furthermore, note that by the Schur Complement, \(X_2 \succeq 0, 1 \succ 0 \implies Y_2 \succeq y_2y_2^T.\) As a result,
\[X'= \begin{bmatrix}
V & vy_2^T
\end{bmatrix}^T \begin{bmatrix}
V & vy_2^T
\end{bmatrix} + \begin{bmatrix}
0 & 0 \\ 0 & Y_2-y_2y_2^T
\end{bmatrix} \succeq 0.\]
Finally, Let \(V_2\) be a matrix with \(V_2^TV_2=Y_2-y_2y_2^T\). Then note
\[X'= \begin{bmatrix}
V & vy_2^T \\ 0 & V_2
\end{bmatrix}^T \begin{bmatrix}
V & vy_2^T \\ 0 & V_2
\end{bmatrix}^T.\]
Since such a representation is an upper block diagonal matrix, we know \(\rank X'= \rank V + \rank V_2 = \rank(V^TV)+\rank(V_2^TV_2) = \rank(X_1)+ \rank(Y_2-y_2y_2^T)\). By properties of the Schur Complement, we thus have \(\rank X'= \rank X_1 + \rank X_2 - 1\).

Thus, by \Cref{thm:optcondition}, \((X',S')\) satisfies all conditions of optimality.
\end{proof}

\begin{proof}[Proof of \Cref{cor:vtxsumranks}]
We now have a solution \(X'\) to \(L(G')\) of rank \(\rank X_1 + \rank X_2 - 1\). Consider \(X'\) in the form \(X'=\begin{bmatrix}
Y_1 & y_1 & y_1y_2^T \\
y_1^T & 1 & y_2^T \\
y_2y_1^T & y_2 & Y_2
\end{bmatrix}\) again. The optimization problem to solve for \(X'\) involves maximizing the dot of \(X'\) and \(L(G')\). Thus, modifying any entry of the upper right or lower left blocks in the matrix will not change the optimal value. It is known that we are able to complete these remaining values of the block matrix to a semidefinite matrix of rank equal to \(\max(\rank X_1, \rank X_2)\), which gives the same objective value as \(X'\) and is thus still optimal. 
\end{proof}

This corollary means that any graph which can be expressed as a vertex sum of two cliques can be analyzed in terms of its two summed components, and this analysis is no harder than solving the max-cut problem simply for the two components. Thus, we can reduce our analysis to biconnected graphs.

\subsection{Edge sum of two triangles}

\begin{proof}[Proof of \Cref{thm:diamond1}]

The rank~1 solutions to the max-cut SDP on the smaller subgraphs give optimal solutions for the max-cut problem on such graphs. By \Cref{cor:uniquecycsol}, the optimal solutions are unique. Thus, the optimal solution to the max-cut problem on the graph \(G\) is unique and agrees with the solutions on the subgraphs \(G_1\) and \(G_2\).

Suppose now that \(w_{23} \le \min(\frac{1}{1/w_{12}+1/w_{13}},\frac{1}{1/w_{24}+1/w_{34}})\). Let \(G'_1\) be the graph on \(4\) vertices with \(w'_{12}=w_{12},w'_{13}=w_{13},\) and \(w'_{23}=\frac{1}{2}w_{23},\) and all other weights \(0\). Similarly, let \(G'_2\) with \(w'_{24}=w_{24},w'_{34}=w_{34},\) and \(w'_{23}=\frac{1}{2}w_{23},\) and all other weights \(0\).  Then, by \Cref{thm:cyclerank1}, the vector \(x=(-1,1,1,-1)\) is an optimal solution to the max-cut problem on both \(G'_1\) and \(G'_2\). By \Cref{cor:uniquecycsol}, this solution is unique on both \(G'_1\) and \(G'_2\). Thus, the matrix \(xx^T\) is indeed the unique optimal solution for the max-cut SDP on the full graph \(G\).

On the other hand, if we know that \(xx^T\) is a solution to the max-cut SDP for \(G\) where \(x=(-1,1,1-1)\), then \(x\) is in fact a solution to the max-cut problem on each of \(G,G_1,\) and \(G_2\). Since we are given that \(G_1\) and \(G_2\) are both rank~1, then by \Cref{thm:cyclerank1}, one of the weights \(w_{12},w_{13}\) and \(w_{23}\) has a reciprocal greater than the sum of the reciprocals of the other two terms. \Cref{cor:uniquecycsol} tells us that this is the unique weight such that \(w_{ij}w_iw_j\) is positive. Since all \(w_{ij}\) are positive and \(x_2x_3=1\), this weight is \(w_{23}\). As a result,
\[\frac{1}{w_{23}}<\frac{1}{w_{12}}+\frac{1}{w_{13}}.\]

Now, suppose that \(xx^T\) is a primal solution, where \(x=(-1,1,1,-1)\). Then the dual solution \(S^*\) is the unique dual solution which agrees with \(L(G,w)\) on the off-diagonal values, and satisfies \(S^*x=0\). Note that
\[S^*=\begin{bmatrix}
    S_1 & 0 \\
    0 & 0
    \end{bmatrix}+
    \begin{bmatrix}
    0 & 0 \\
    0 & S_2
     \end{bmatrix}+\frac14
    \begin{bmatrix}
    0 & 0 & 0 & 0 \\
    0 & -w_1 & w_1 & 0 \\
    0 & w_1 & -w_1 & 0 \\
    0 & 0 & 0 & 0
    \end{bmatrix}.\]
\end{proof}
satisfies both of these conditions, and thus must be the unique dual solution.

\begin{proof}[Proof of \Cref{thm:diamond2}]
Suppose for the sake of simplicity that \(w_{12}>w_{13}\) and \(w_{24}>w_{34}\). Then first assume the matrix \(X=xx^T\), where \(x=(\epsilon_1,-1,1, \epsilon_2)\), is a solution. Repeating the argument from the previous proof, the vector \(x\) must be a solution to the max-cut problem for each of \(G,G_1\) and \(G_2\). By exhaustively checking all cases to solve the max-cut problem on \(G_1\) and \(G_2\), we find that \(x=(1,-1,1,1)\). Now, the dual matrix has the form \(S'=L(G,\bar{w})\), where \(\bar{w}_{ij}=-x_ix_jw_{ij}\) by \Cref{lem:rankone}. The upper 3 by 3 principal submatrix of \(S'\) has entries
\[
\begin{bmatrix}
w_{12}-w_{13} & -w_{12} & w_{13} \\
-w_{12} & w_{12}+w_{23}+w_{14} & -w_{23} \\
w_{13} & -w_{23} & -w_{13}+w_{23}-w_{34} \\
\end{bmatrix}.\]
The determinant of this matrix simplifies to
\[w_{23}(w_{12}-w_{13})(w_{24}-w_{34})-w_{12}w_{13}w_{24}+w_{12}w_{13}w_{34}-w_{12}w_{23}w_{24}+w_{13}w_{14}w_{24}.\]
Since \(S'\) is positive semidefinite, then this determinant must be nonnegative. For the determinant to be positive, \(w_{23}\) must satisfy
\[w_{23} \ge \frac{1}{1/w_{12}-1/w_{13}}+\frac{1}{1/w_{24}-1/w_{34}},\]
as desired.

On the other hand, if the weights satisfy both inequalities given in the first condition, then we can see for some \(\epsilon_1,\epsilon_2 \in \{-1,1\}\) it is true that when \(x=(\epsilon_1,-1,1, \epsilon_2)\), then \(X=xx^T\) is an optimal solution to any graph on \(4\) vertices such that it has \(3\) edges which form a triangle and the weights satisfy the condition in \Cref{thm:cyclerank1}. Since we can find positive real numbers \(p\) and \(q\) such that \(p+q=w_{23}\), \(p \ge \frac{1}{1/w_{12}-1/w_{13}}\), and \(q \ge \frac{1}{1/w_{24}-1/w_{34}}\). Thus, since \(X\) is an optimal solution to both of these graphs, it is an optimal solution for the large graph as well.

Note that again, similar to above, the existence of the primal optimal solution allows us to exactly find the dual solution of the matrix. We can check that
\[S^*=\begin{bmatrix}
    S_1 & 0 \\
    0 & 0
    \end{bmatrix}+
    \begin{bmatrix}
    0 & 0 \\
    0 & S_2
    \end{bmatrix}+\frac14
    \begin{bmatrix}
    0 & 0 & 0 & 0 \\
    0 & w_1 & w_1 & 0 \\
    0 & w_1 & w_1 & 0 \\
    0 & 0 & 0 & 0
    \end{bmatrix}\]
indeed agrees with the Laplacian matrix on the off-diagonal values and that \(S^*x=0\), thus it must be the unique dual solution.
\end{proof}

\section{Future work}\label{sec:future work}

In our project, we began by considering a vertex sum, or what happened when we joined two subgraphs at a vertex. We were able to fully characterize a primal and dual solution to the SDP given solutions to the SDPs for the subgraphs. Doing the same for an edge sum, or joining two graphs at an edge, proved to be much more difficult. We were able to show a result for joining two \(K_3\)s given a constraint on the rank. Future work may look to extend our results to the max-cut SDP for general edge sums or larger clique sums. To account for more general clique sums, one possible direction to look at is to add two graphs. In particular, the weight of an edge in this sum would be the sum of the corresponding edge weights of the summand graphs. This area is promising because it gave a good characterization for the theorem on the butterfly graph and for vertex sums.

A conjecture we have proposed for the solutions to the max-cut SDP of the edge sum of two graphs is as follows.
\begin{conjecture}
Let \(G_1\) and \(G_2\) be two graphs with vertex sets \(V_1=\{1,2,\ldots, m+1, m+2\},\) and \(V_2=\{m+1,m+2,\ldots, n-1, n\}.\) Let \((X_1,S_1)\) and \((X_2,S_2)\) be the primal-dual solution pairs to the max-cut SDPs on \(G_1\) and \(G_2\), respectively.

Let \(G\) be the edge sum of \(G_1\) and \(G_2,\) with common edge \((m+1,m+2).\) Let \((X^*,S^*)\) be the primal and dual optimal solutions to the graph.

Then, the following two statements are equivalent:
\begin{enumerate}
    \item \(X_1, X_2\) are rank~1 matrices and agree on the intersection.
    \item For some common choice of \(\pm\) on the off-diagonal \(w_1\)s, we have  \[S=\begin{bmatrix}
    S_1^* & 0 \\
    0 & 0
    \end{bmatrix}+
    \begin{bmatrix}
    0 & 0 \\
    0 & S_2^*
    \end{bmatrix}-\frac14
    \begin{bmatrix}
    0 & 0 & 0 & 0 \\
    0 & w_1 & \ \pm w_1 & 0 \\
    0 & \pm w_1 & w_1 & 0 \\
    0 & 0 & 0 & 0
    \end{bmatrix}.
    \]
    
    Here, the first two matrices are \(n\times n.\) The matrix \(S_1\) is \(m+2 \times m+2,\) and the matrix \(S_2\) is \(n-m \times n-m.\) In the last matrix, the \(w_1\)s are in entry \(A_{m+1,m+1}\) and \(A_{m+2,m+2}.\)
\end{enumerate}
\end{conjecture}

In addition, our analysis can be extended in the future to cover more families of graphs and operations upon them. In particular, we believe that series-parallel graphs are an interesting family of graphs to investigate. Series-parallel graphs can easily be decomposed into vertex sums, which means their rank one solutions may be able to be characterized.

\section{Acknowledgements}
We would like to thank the MIT PRIMES program for creating this opportunity for us to come together as a research group, without which this experience would not have been possible for any of us. We would also like to offer a special word of thanks to our research mentor, Diego Cifuentes, who provided us with invaluable resources and dedicated many hours to helping us make the most of this research experience.

\bibliographystyle{abbrv}
\bibliography{primes}

\begin{thebibliography}{10}

\bibitem{alizadeh1997complementarity}
F.~Alizadeh, J.-P.~A. Haeberly, and M.~L. Overton.
\newblock Complementarity and nondegeneracy in semidefinite programming.
\newblock {\em Mathematical programming}, 77(1):111--128, 1997.

\bibitem{avidor2005rounding}
A.~Avidor and U.~Zwick.
\newblock Rounding two and three dimensional solutions of the {SDP} relaxation
  of {MAX CUT}.
\newblock In {\em Approximation, Randomization and Combinatorial Optimization.
  Algorithms and Techniques}, pages 14--25. Springer, 2005.

\bibitem{baharona1988applications}
F.~Baharona, M.~Grötschel, M.~Jünger, and G.~Reinelt.
\newblock An application of combinatorial optimization to statistical physics
  and circuit layout design.
\newblock {\em Operations Research}, 36(3):377--513, 1988.

\bibitem{blekherman2012semidefinite}
G.~Blekherman, P.~A. Parrilo, and R.~R. Thomas.
\newblock {\em Semidefinite optimization and convex algebraic geometry}.
\newblock SIAM, 2012.

\bibitem{burer2002rank}
S.~Burer, R.~D. Monteiro, and Y.~Zhang.
\newblock Rank-two relaxation heuristics for max-cut and other binary quadratic
  programs.
\newblock {\em SIAM Journal on Optimization}, 12(2):503--521, 2002.

\bibitem{silva2018strict}
M.~K. de~Carli~Silva and L.~Tun{\c{c}}el.
\newblock Strict complementarity in semidefinite optimization with elliptopes
  including the maxcut {SDP}.
\newblock {\em SIAM Journal on Optimization}, 29(4):2650--2676, 2019.

\bibitem{de2020integrality}
F.~M. de~Oliveira~Filho and F.~Vallentin.
\newblock On the integrality gap of the maximum-cut semidefinite programming
  relaxation in fixed dimension.
\newblock {\em Discrete Analysis}, (10):1--17, 2020.

\bibitem{delorme1993laplacian}
C.~Delorme and S.~Poljak.
\newblock Laplacian eigenvalues and the maximum cut problem.
\newblock {\em Mathematical Programming}, 62(1-3):557--574, 1993.

\bibitem{goemans1995improved}
M.~X. Goemans and D.~P. Williamson.
\newblock Improved approximation algorithms for maximum cut and satisfiability
  problems using semidefinite programming.
\newblock {\em Journal of the ACM (JACM)}, 42(6):1115--1145, 1995.

\bibitem{karp1972reducibility}
R.~M. Karp.
\newblock Reducibility among combinatorial problems.
\newblock In {\em Complexity of computer computations}, pages 85--103.
  Springer, 1972.

\bibitem{khot2007optimal}
S.~Khot, G.~Kindler, E.~Mossel, and R.~O’Donnell.
\newblock Optimal inapproximability results for {MAX-CUT} and other 2-variable
  {CSP}s?
\newblock {\em SIAM Journal on Computing}, 37(1):319--357, 2007.

\bibitem{laurent1995positive}
M.~Laurent and S.~Poljak.
\newblock On a positive semidefinite relaxation of the cut polytope.
\newblock {\em Linear Algebra and its Applications}, 223(224):439--461, 1995.

\bibitem{laurent1996facial}
M.~Laurent and S.~Poljak.
\newblock On the facial structure of the set of correlation matrices.
\newblock {\em SIAM Journal on Matrix Analysis and Applications},
  17(3):530--547, 1996.

\bibitem{nagy2014forbidden}
M.~Nagy, M.~Laurent, and Varvitsiotis.
\newblock Forbidden minor characterizations for low-rank optimal solutions to
  semidefinite programs over the elliptope.
\newblock {\em Journal of Combinatorial Theory, Series B}, 108:40--80, 2014.

\bibitem{pataki1998rank}
G.~Pataki.
\newblock On the rank of extreme matrices in semidefinite programs and the
  multiplicity of optimal eigenvalues.
\newblock {\em Mathematics of operations research}, 23(2):339--358, 1998.

\bibitem{vandenberghe1996semidefinite}
L.~Vandenberghe and S.~Boyd.
\newblock Semidefinite programming.
\newblock {\em SIAM review}, 38(1):49--95, 1996.

\bibitem{wolkowicz2012handbook}
H.~Wolkowicz, R.~Saigal, and L.~Vandenberghe.
\newblock {\em Handbook of semidefinite programming: theory, algorithms, and
  applications}, volume~27.
\newblock Springer Science \& Business Media, 2012.

\end{thebibliography}
\end{document}